\documentclass[letterpaper,11pt,twoside,reqno]{amsart}
\usepackage{times}
\usepackage{amsmath}
\usepackage{amssymb}
\usepackage{amsfonts}
\usepackage{amscd}
\usepackage{amsthm}
\usepackage{amsxtra}
\usepackage{stmaryrd}

\usepackage[all]{xy}  \CompileMatrices
\usepackage{mathrsfs}
\usepackage{nicefrac}
\usepackage{graphicx}
\usepackage{color}
\usepackage{enumerate}
\usepackage{wrapfig}

\usepackage{algorithm}

\renewcommand{\paragraph}[1]{\par\vspace{1ex}\noindent #1}

\theoremstyle{plain}
\newtheorem{theorem}{Theorem}

\newtheorem*{proposition*}{Proposition}

\newtheorem*{corollary*}{Corollary}
\newtheorem{lemma}[theorem]{Lemma}

\newtheorem*{theorem*}{Theorem}
\newtheorem*{lemma*}{Lemma}
\newtheorem*{conjecture*}{Conjecture}

\newtheorem*{question*}{Question}

\theoremstyle{definition}

\newtheorem*{exercise*}{Exercise}

\theoremstyle{remark}

\newtheorem*{remark*}{Remark}
\newtheorem{remsTh}[theorem]{Remarks}

\newcommand{\subclass}[1]{}

\newcommand{\enumTi}[1]{\renewcommand{\theenumi}{#1}}

\newcommand{\alphenumi}{\enumTi{\alph{enumi}}}

\newcommand{\romenumi}{\enumTi{\roman{enumi}}}

\renewcommand{\em}{\sl}

\newcommand{\sstack}[1]{{\substack{#1}}}

\newcommand{\sabs}[1]{{\lvert{#1}\rvert}}

\DeclareMathOperator{\pr}{pr}

\newcommand{\nfrac}[2]{{\nicefrac{#1}{#2}}}

\newcommand{\RR}{\mathbb{R}}

\DeclareMathOperator{\Prb}{\mathbf{P}}

\newlength{\algotabbingwidth}
\begin{document}

\title{Random lifts of $K_5\setminus e$ are 3-colourable}%
\author{Babak Farzad}%
\address{BF: Mathematics Department\\
  Brock University\\
  St.~Catharines, Canada%
}%
\email{bfarzad@brocku.ca}%
\author{Dirk Oliver Theis}%
\address{DOT: Fakult\"at f\"ur Mathematik\\
  Otto-von-Guericke-Universit\"at Magdeburg\\
  Magdeburg, Germany
}%
\email{theis@ovgu.de}%
\thanks{This work was supported by an NSERC International Collaborative Funding Initiative Grant.  The work of the first author is partially supported by
    NSERC Discovery Grant \# 356035-08.  The work of the second author was supported by the Fonds National de la Recherche Scientifique (F.R.S.--FNRS)}

\begin{abstract}
  Amit, Linial, and Matou\v sek (Random lifts of graphs III: independence and chromatic number, \textit{Random Struct.\ Algorithms}, 2001) have raised the
  following question: Is the chromatic number of random $h$-lifts of $K_5$ asymptotically (for $h\to\infty$) almost surely (a.a.s.)\ equal to a single number?
  In this paper, we offer the following partial result: The chromatic number of a random lift of $K_5\setminus e$ is a.a.s.\ three.
\end{abstract}

\maketitle

\section{Introduction}

Let $G$ be a graph, and $h$ a positive integer.  An \textit{$h$-lift} of $G$ is a graph $\widetilde G$ which is an $h$-fold covering of $G$ in the topological
sense.  Equivalently, there is a graph homomorphism $\phi\colon \widetilde G\to G$ which maps the neighbourhood of any vertex $v$ in $\widetilde G$ one-to-one
onto the neighbourhood of the vertex $\phi(v)$ of $G$.  The graph $G$ is called the \textit{base graph} of the lift.

More concretely, we may say that an $h$-lift of $G$ has vertex set $V(G)\times[h]$ (where we let $[h]:=\{1,\dots,h\}$ as usual).  The set $\{v\}\times[h]$ is
called the \textit{fibre over $v$}.  Fixing an orientation of the edges of $G$, the edge set of an $h$-lift is of the following form: There exist permutations
$\sigma_e$ of $[h]$, $e\in E(G)$, such that for every two adjacent vertices $u$ and $v$ of $G$, if the edge $uv$ is oriented $u\to v$, the edges between the fibres
$\{v\}\times[h]$ and $\{u\}\times[h]$ are $(u,j)(v,\sigma_{uv}(j))$, $j\in[h]$.
Changing the orientation of the edges in the graph does not change the lift, provided that permutations on edges on which the orientation is changed are
replaced by their respective inverses.  In this spirit, for an edge $uv$ in $G$, regardless of its orientation, we denote by $\sigma_{uv}$ the permutation for
which the edges between the fibres are $\{ (u,j)(v,\sigma_{uv}(j)) \mid j\in[h] \}$.

By a \textit{random $h$-lift} we mean a graph chosen uniformly at random from the graphs just described, which amounts to choosing a permutation, uniformly at random,
independently for every edge of $G$.

Random lifts of graphs have been proposed in a seminal paper by Amit, Linial, Matou\v sek, and Rozenman \cite{AmitLinialMatousekRozenman01}.  Their paper
sketched results on connectivity, independence number, chromatic number, perfect matchings, and expansion of random lifts, and was followed by a series of
papers containing broader and more detailed results by the same and other authors
\cite{AmitLinialNathan02,AmitLinialNathan06,AmitLinialMatousek01,LinialRozenman02}, and e.g.\ \cite{BiluLinial06,DrierLinial06},
\cite{BurginCheboluCooperFrieze06}.

In \cite{AmitLinialMatousek01} Amit, Linial, and Matou\v sek focused on independence and chromatic numbers of random lifts of graphs.  They asked the following
question.
\begin{quote}
  Is there a zero-one law for the chromatic number of random lifts?  In particular, is the chromatic number of a random lift of $K_5$ a.a.s.\ (for $h\to\infty$)
  equal to a single number (which may be either 3 or~4)?
\end{quote}

A random $h$-lift $\widetilde G$ of $K_5$ a.a.s.\ has an odd cycle, whence a.a.s.\ we have $\chi(\widetilde G) \geq 3$.  Moreover, $\widetilde G$ a.a.s.\ does not
contain a 5-clique.  Brooks' theorem implies that a.a.s.\ $\chi(\widetilde G)\le 4$.  So, a.a.s.\ $\chi(\widetilde G) \in \{3,4\}$.

In their paper, Amit, Linial, and Matou\v sek \cite{AmitLinialMatousek01} conjectured that the chromatic number of random lifts of any fixed base graph obeys a
zero-one law, i.e., it is asymptotically almost surely equal to a fixed number (depending only on the base graph).  In the case when the base graph is $K_n$,
they prove that $\chi(\widetilde G) = \Theta(n/\log n)$ a.a.s.\ (the constant in the $\Theta$ notation may depend neither on $h$ nor on $n$).  Five is the
smallest value for $n$, for which this is not trivial.

In this paper, we contribute the following to this problem.

\begin{theorem}\label{thm:K_5-e}
  A random lift of $K_5\setminus e$ is a.a.s.\ 3-colorable.
\end{theorem}

\section{Notation and Terminology}

Let $G:= K_5\setminus e$.  Clearly, $G$ is obtained by joining a cycle $C:=[x_1,x_2,x_3]$ to a stable set $S:=\{y_1,y_2\}$. Here, by \textsl{join} we mean that
every vertex of $C$ is made adjacent to every vertex of $S$.  From now on, $\widetilde G$ will be a random $h$-lift of $G$.  Let $\widetilde G_C$ and
$\widetilde G_S$ denote the subgraphs of $\widetilde G$ induced by the fibres over the vertices of $C$ and those over vertices of $S$, respectively.
Moreover, for $x\in V(G)$, we denote by $V_x = \{x\}\times [h]$ the set of vertices of $\widetilde G$ over $x$.
Similarly, for any set $U$ of vertices of $\widetilde G$ and $x\in V(G)$, we let $U_x := U\cap V_x$.


As an \textit{hors d'\oe uvre} intended to familiarise the reader with the most basic random lift arguments, we serve the following easy lemma.

\begin{lemma}\label{lem:ub-no-cycles}
  The graph $\widetilde G_C$ is a union of cycles, each of which is divisible by three.  A.a.s., the number of cycles in $\widetilde G_C$ is at most $\log^2 h$.
\end{lemma}
\begin{proof}
  The cycles with length $3\ell$ of $\widetilde G_C$ correspond to the cycles with length $\ell$ of the permutation
  $\sigma_{x_1x_2}\circ\sigma_{x_2x_3}\circ\sigma_{x_3x_1}$.  The latter is a uniformly distributed random permutation of $[h]$.  It is a folklore fact (e.g.,
  \cite{LovaszCombProbExe}) that the average number of cycles of a random permutation of $[h]$ is $\log h + o(1)$.  The statement of the lemma now follows from
  Markov's inequality.
\end{proof}

Lemma~\ref{lem:ub-no-cycles} allows us to assume that
$\widetilde G_C$ has at most $\log^2h$ cycles.  As a matter of fact, this is the only statement about $\widetilde G_C$ which we need.

\section{The 3-colouring algorithm}\label{sec:the-alg}

\begin{algorithm}[htbp]
  \caption{\label{alg:3colalg}Three-Colour $\widetilde G$}%
  \begin{minipage}[c]{1.0\linewidth}
    \flushleft%
    \textbf{Phase I:}
    \begin{enumerate}[(1)]
    \item\label{enum:mainalg:start}%
      The algorithms starts with all edges in $\widetilde G_C$ exposed, but no edge in between $\widetilde G_C$ and $\widetilde G_S$ exposed.  If $\widetilde
      G_C$ has more than $\log^2h$ cycles, \textbf{fail.}
    \item\label{enum:mainalg:killCcycles}%
      Choose exactly one red vertex in each cycle of $\widetilde G_C$. 
    \end{enumerate}
    \textbf{Phase II:}
    \begin{enumerate}[(1)]\setcounter{enumi}{2}
    \item\label{enum:mainalg:setup-loop}%
      Expose all edges incident to red vertices.  If there exists a vertex in $\widetilde G_S$ which has two or more red neighbours, \textbf{fail.}  Otherwise,
      let $P(0)$ be the set of pale vertices before the first iteration.
    \item\label{enum:mainalg:loop}%
      For $t=1,\dots, \lfloor h^{\nfrac13}\rfloor$:
      \begin{enumerate}[({\ref{enum:mainalg:loop}}.1)]
      \item Let $v$ be chosen arbitrarily from the set $P(t-1)$.
      \item From the two non-exposed edges incident to $v$, expose one arbitrarily (the other edge remains unexposed).  Let $u$ be the end-vertex in $\widetilde
        G_C$ of the exposed edge.
      \item Expose the other edge incident to $u$, and let $v'$ be the corresponding neighbour of $u$ in $\widetilde G_S$.  If $v' \in \bigcup_{s=0}^{t-1}
        P(s)$, \textbf{fail.} Otherwise $P(t) = P(t-1)\cup\{v'\}\setminus\{v\}$ (this is now the new set of pale vertices).
      \item Colour $u$ red.
      \end{enumerate}
    \end{enumerate}
    \textbf{Phase III:}
    \begin{enumerate}[(1)]\setcounter{enumi}{4}
    \item\label{enum:mainalg:expose-all}%
      Expose all remaining edges.
    \item\label{enum:mainalg:remaining-red}%
      Colour every vertex red which is in $\widetilde G_S$ and does not have a red neighbour.
    \item\label{enum:mainalg:bipartite}%
      If the graph induced by the non-red vertices is 
      acyclic, colour it black and white, otherwise \textbf{fail.}
    \end{enumerate}
  \end{minipage}
\end{algorithm}

Our colouring algorithm is detailed in the box Algorithm~\ref{alg:3colalg}.  We use the colours \textit{red,} \textit{black,} and \textit{white,} where the
colour red will have a special significance.  We point the reader to the fact that, once Algorithm~\ref{alg:3colalg} has coloured a vertex, the vertex never
changes its colour or becomes uncoloured again.
A vertex of $\widetilde G_S$ which is adjacent to precisely one red vertex is called \textit{pale} (this is not a colour).

The algorithm works in three phases.  In phase~I, Steps~(\ref{enum:mainalg:start}--\ref{enum:mainalg:killCcycles}), we destroy the uncoloured cycles of
$\widetilde G_C$ by colouring one vertex per cycle red.  By Lemma~\ref{lem:ub-no-cycles}, a.a.s., we colour at most $\log^2h$ vertices red in Phase~I, i.e.,
Phase~I fails with probability $o(1)$.

In Phase~II, more accurately in the loop~(\ref{enum:mainalg:loop}), the algorithm successively chooses uncoloured vertices of $\widetilde G_C$ and colours them
red.  This is done by maintaining the set $P(\cdot)$ of pale vertices (i.e., those vertices of $\widetilde G_S$ which are adjacent to precisely one red vertex).

In Phase~III, Steps~(\ref{enum:mainalg:expose-all}--\ref{enum:mainalg:bipartite}), the remaining vertices are coloured in a straight forward way.

The rationale behind the algorithm is as follows.
  
At any fixed time between Steps~(\ref{enum:mainalg:setup-loop}) and~(\ref{enum:mainalg:expose-all}), consider the connected components of $\widetilde G_C$ after
deleting all red vertices.  These are uncoloured paths of different lengths in $\widetilde G_C$, separated by red vertices.  We call them \textit{chunks.}
These chunks can be thought of as the vertices of a multi-graph, which we call the \textit{chunk-graph}, whose edges are the pale vertices in $\widetilde G_C$:
Every pale vertex has precisely two uncoloured neighbours in $\widetilde G_C$, thus connecting the corresponding chunks.  We refer to such a connection between
chunks via a pale vertex as a \textit{chunk-edge.}  A chunk-edge may be a loop, which happens when a pale vertex have both uncoloured neighbours in the same
chunk.  Furthermore, there may be parallel chunk-edges in the chunk-graph, which happens when two pale vertices connect the same pair of chunks.
The reason why, in Step~\ref{enum:mainalg:setup-loop} of the algorithm, we abort if a vertex has two or more red neighbors, is only because such vertices would
not correspond to edges of the chunk-graph.  Indeed, at the end of Phase~II, there are only two kinds of uncolored vertices left: Those making up the chunk
graph, and those being colored red in Step~\ref{enum:mainalg:remaining-red}.

The chunk-graph is a random multi-graph.  At Step~(\ref{enum:mainalg:setup-loop}), it has as many vertices as there are cycles in $\widetilde G_C$ (at most
$\log^2h$ by Lemma~\ref{lem:ub-no-cycles}), and as many edges as there are pale vertices.
If the algorithm does not fail in Step~(\ref{enum:mainalg:setup-loop}), then to every red vertex there are two pale vertices, and they are all distinct.
Hence, at this time, there are twice as many chunk-edges as there are chunks.


When the algorithm proceeds through loop~(\ref{enum:mainalg:loop}), the number of chunks is increased as we colour more vertices of $\widetilde G_C$ red.
However, the number of pale vertices stays constant, and hence so does the number of chunk-edges.

The reasoning at this point is a heuristic analogy with the random (simple) graph model $G(n,m)$, where a set of $m$ edges is drawn uniformly at random from the
set of all possible $m$-sets of edges between $n$ vertices.  For us, $n$ is the number of chunks and $m$ is the number of chunk-edges.  At
Step~(\ref{enum:mainalg:setup-loop}), where $m=2n$, we expect the chunk-graph to contain lots of cycles (including loops and parallel edges), which makes it
unlikely that it can be coloured with just the two remaining colours.  However, when $n$ grows and $m$ stays constant, a random graph $G(n,m)$ will be acyclic
as soon as $m \ll n$, and we expect the same to be true for the chunk-graph.

There are complications in making this heuristic analogy work rigorously, the foremost being that the distribution of the edges in the chunk-graph is not
uniform but instead depends on the sizes of the chunks.  We will address these issues in the next section.


\section{Proof of correctness of the 3-colouring algorithm}

We prove that a.a.s.\ Algorithm~\ref{alg:3colalg} properly 3-colours $\widetilde G$.

\begin{lemma}\label{lem:alg-doesnt-fail-trivially}
  A.a.s., Algorithm~\ref{alg:3colalg} does not fail in Steps~(\ref{enum:mainalg:start}), (\ref{enum:mainalg:setup-loop}), or~(\ref{enum:mainalg:loop}.3).
\end{lemma}

\begin{proof}
Lemma~\ref{lem:ub-no-cycles} implies that, a.a.s., 
   the algorithm does not fail in Step~(\ref{enum:mainalg:start}).

For Step~(\ref{enum:mainalg:setup-loop}), note that, 
   at this point in the algorithm, the probability that a fixed vertex in $\widetilde G_S$ has two or more
   red neighbours is $O( \nfrac{(\log^4h)}{h^2})$.  
Hence, the probability that there exists such a vertex having two or more red neighbours is
  $O(\nfrac{(\log^4h)}{h}) = o(1)$.
  
For Step~(\ref{enum:mainalg:loop}.3), we see that for each fixed $t$, 
   the probability that $v'\in \bigcup_{s=0}^{t-1} P(s)$ is $O(h^{-\nfrac23})$.  
Thus, the probability that the algorithm fails after at most $t$ iterations is $O(t h^{-\nfrac23})$.  Consequently, the probability that the algorithm fails at
  Step~(\ref{enum:mainalg:loop}.3) before completing $t:=\lfloor h^{\nfrac13} \rfloor$ iterations is $o(1)$.
\end{proof}

Denote by $T$ the last iteration (value of $t$) of the loop~(\ref{enum:mainalg:loop}) which is completed (without failing).
We let $R(t)$, $t=0,1,\dots,T$ be the set of vertices which are red after $t$ iterations of the loop~(\ref{enum:mainalg:loop}).  In particular, $R(0)$ is the
set of vertices coloured red in Step~(\ref{enum:mainalg:killCcycles}).  Let $R^+(t) := R(t)\setminus R(0)$.  Recall that adding an index to a letter denoting a
set refers to taking its intersection with the corresponding fibre, for example $R_x(t)$ refers to $V_x\cap R(t)$.
Moreover, we use the following notation to refer to the cardinalities of each of these sets: If a set is denoted by an upper-case letter (possibly with sub- or
superscript or followed by parentheses), the corresponding lower-case letter (with the same sub- or superscripts or parentheses) denotes its cardinality.  For
example $r_x(t) = \sabs{R_x(t)}$.  We have the following.

\begin{lemma}\label{lem:red-Rx-indep-uniform}
  For each $x\in C$ and $t=1,\dots,T$, set $R^+_x(t)$ is uniformly distributed in the set of all $(r^+_x(t))$-element subsets of $V_x\setminus R_x(0)$.
\end{lemma}
\begin{proof}
%
  Fix an $x \in C$.  In every iteration of the loop~(\ref{enum:mainalg:loop}) in which the fibre over $x\in C$ is selected in Step~(\ref{enum:mainalg:loop}.3),
  when exposing the edge in Step~(\ref{enum:mainalg:loop}.3), the vertex $u$ is selected uniformly at random from the set of all previously uncoloured vertices
  in $V_x$.
   In other words, for every fixed value of $R^+_x(t-1)$, the distribution of $u$ is uniform.  By induction, $R^+_x(t)$ is uniformly distributed.
\end{proof}

\begin{lemma}\label{lem:red-not-adj}
  In the loop~(\ref{enum:mainalg:loop}) of Algorithm~\ref{alg:3colalg}, a.a.s.\ no two adjacent vertices are coloured red.
\end{lemma}

\begin{proof}
  Let $x_1,x_2\in C$, and consider the situation after $T$ iterations, i.e., when the algorithm leaves the loop~(\ref{enum:mainalg:loop}).  By
  Lemma~\ref{lem:red-Rx-indep-uniform}, at this time, the expected number of edges between $V_{x_1}$ and $V_{x_2}$ both of whose end vertices are red is at most
  \begin{equation*}
    h\cdot \frac{T}{h-r_{x_1}(0)} \cdot \frac{T}{h-r_{x_2}(0)}
    \quad=\quad
    O\biggl( \frac{h^{\nfrac53}}{\bigl( h-\log^2 h \bigr)^2} \biggr) 
    \;=\;
    o(1).
  \end{equation*}
\end{proof}

Now, it only remains to show that when Step~(\ref{enum:mainalg:bipartite}) of Algorithm~\ref{alg:3colalg} is reached, the graph consisting of the yet uncoloured
vertices is a.a.s.\ acyclic.  

Now, suppose that the algorithm has completed Phase~II without failing, i.e., we find ourselves just before Step~(\ref{enum:mainalg:expose-all}).
Let $H$ denote the chunk graph as we defined in Section~3.  Thus $H$ is a random multi-graph with $n \le r(T) = T + r(0) = \Theta( h^{\nfrac13} )$ vertices and
$m := p(T) = 2r(0) = O(\log^2 h)$ edges.  In fact, if no two red vertices are adjacent, the first inequality becomes an equation,
cf.~Lemma~\ref{lem:red-not-adj}.  The distribution of $H$ can be described in terms of random permutations taking into account the edges which have already been
exposed, and the sizes of the chunks.  It appears sensible to guess that $H$ has no cycles.  That is in fact correct.

\subsubsection*{\bf Sizes of the chunks}
The first thing we require to turn this analogy into a rigorous proof is an upper bound on the sizes of the chunks.  We find it convenient to reduce the
question to the distribution of the gaps between $n$ points drawn uniformly at random from the interval $[0,1]$.  There, the probability that two consecutive
points enclose a gap of size $a$ is $(1-a)^n$, which yields an upper bound of, say, $\nfrac{(2h\log n)}{n}$ for the largest gap, a.a.s.
In the following lemmas, we put this plan into action.

Let $n$ numbers $Y_1,\dots,Y_n$ be drawn independently uniformly at random from $[N]$, where $N$ is a function of $n$.  Let $S_k$ be the $k$-th order
statistics (i.e., $0\le S_1\le\dots\le S_n\le 1$, and $\{S_1,\dots, S_n\} = \{Y_1,\dots,Y_n\}$) and set $S_0 := 0$ and $S_{n+1} := N$.

We determine the distribution of $S_{k+1}-S_k$.  This can be done directly, but it can also easily be derived from the Bapat-Beg theorem, of which the following
is a special case (see the appendix for a proof).

\begin{lemma}\label{lem:distrib-gaps-points-interval}
  Let $X_1,\dots,X_n$ be points drawn independently uniformly at random in $[0,1]$ and denote by $S'_k$ the $k$-th order statistics.  With $S'_0 := 0$ and
  $S'_{n+1} := 1$, 
  for each $k=0,\dots,n$, the distribution of $S'_{k+1}-S'_k$ is as follows: $\Prb[S'_{k+1}-S'_k > a] = (1-a)^n$.
  \qed
\end{lemma}

For the discrete version we obtain the following.

\begin{lemma}\label{lem:distrib-gaps-points-discrete}
  For every $a > 0$, we have
  \begin{equation*}
    \Prb[ S_{k+1} - S_k   >   \tfrac{aN}{n} ]      \le   e^{-a + O(\nfrac{n}{N})},
  \end{equation*}
  (with an absolute constant in the $O(\cdot)$).
\end{lemma}
\begin{proof}
  Let $X_1,\dots,X_n$ be drawn independently uniformly at random from $[0,1]$.  
  We can assume
  that the $Y$s are the $X$s multiplied by $N$
  and then rounded up: $Y_j = \lceil NX_j \rceil$.  We also assume that the permutation taking the $X$s to the $S'$s is equal to the permutation taking the $Y$s
  to the $S$s (this condition makes sense when two $Y$s coincide).
By Lemma~\ref{lem:distrib-gaps-points-interval},
  we conclude that
  \begin{multline*}
    \Prb[ S_{k+1} - S_k   >   \tfrac{aN}{n} ]
    \quad\le\quad
    \Prb[ S'_{k+1} - S'_k   >   (\tfrac{aN}{n}-2)/N ]
    \\
    =
    (1- (\nfrac{a}{n}-\nfrac{2}{N}))^n
    \;\le\;
    e^{-a + \nfrac{2n}{N}}.
  \end{multline*}
\end{proof}

From this, we conclude the following.

\begin{lemma}\label{lem:max-gaps-subset}
  Let an $n$-subset $R$ be drawn uniformly at random from all the $n$-subsets of $[N]$, and $a >0$.
  The probability that there are $\lceil aN/n \rceil$ consecutive numbers not in $R$ is at most $(n+1)e^{-a + O(\nfrac{n}{N})}$.
\end{lemma}
\begin{proof}
  Let $b := \lceil aN/n \rceil$, and let $Y_1,\dots,Y_n$ be drawn independently uniformly at random from $[N]$.  Let $A$ be the event that the $Y_j$'s are all
  distinct, $\bar A$ its complement, and let $B$ be the event that there are $b$ consecutive numbers not containing any of the $Y_j$'s.  Since $\Prb(B)$ is a
  convex combination of $\Prb(B|A)$ and $\Prb(B|\bar A)$, and $\Prb(B)\le (n+1)e^{-a + O(\nfrac{n}{N})}$ by Lemma~\ref{lem:distrib-gaps-points-discrete}, this
  upper bound must also be true for the smaller of the two conditional probabilities.  But, clearly
  $\Prb(B|A) \le \Prb(B|\bar A)$.
\end{proof}

We can now prove the upper bound on the sizes of the chunks.  

\begin{lemma}\label{lem:max-gaps-red}
  Let $\omega \xrightarrow{\;\scriptscriptstyle h\;}\infty$ arbitrarily slowly.  If $n$ is the number of red vertices in $\widetilde G_C$ at the completion of
  Phase~II of the algorithm, a.a.s.\ as $h\to\infty$, there is no chunk with size larger than $6(\omega+\log n)h/n$.
\end{lemma}
\begin{proof}
Choose an arbitrary $x\in C$.  
By Lemma~\ref{lem:red-Rx-indep-uniform}, 
   the conditions of Lemma~\ref{lem:max-gaps-subset} are satisfied if 
   we let $n := r^+_x(T)$ 
   and $N := \sabs{V_x\setminus R_x(0)}$.  
The vertices in $V_x\setminus R_x(0)$ are numbered in the following way.

For each cycle of $\widetilde G_C$, choose an orientation.  
The numbers associated to the vertices in the intersection of $V_x\setminus R_x(0)$ and this cycle
  are then taken consecutively:  
  starting with the vertex in $V_x\setminus R_x(0)$ which, in positive orientation, 
  is next to the $R(0)$-vertex of the cycle, and
  continuing to number in positive orientation.

If there is a path in $\widetilde G_C$ of length greater than 
   $6(\omega+\log n)h/n$ not containing a red vertex, 
   then there is a gap in $[N]$ larger than $(\omega+\log n)N/n$.
(Notice that every third vertex of the path belongs to $V_x$. 
 The factor~2 comes from the left and right end strips, 
 i.e., the vertices which are close to the $R(0)$-vertex on a cycle but which do not have consecutive numbers.) 
 By Lemma~\ref{lem:max-gaps-subset}, the probability of this happening is at most
  \begin{equation*}
    (n+1)\,e^{-\omega-\log n +O(n/N)} = \tfrac{n+1}{n}\, e^{-\omega +O(1)} = o(1).
  \end{equation*}
\end{proof}

\subsubsection*{\bf Bounding the expected number of cycles in $H$}
We now come to the classical first-moment argument which shows that, a.a.s., our random multi-graph $H$ has no cycles.
For the remainder of this section, 
   we condition on the event that the algorithm does not fail before Step~(\ref{enum:mainalg:expose-all}), 
   and that no two adjacent vertices have been coloured red 
   (cf.~Lemmas \ref{lem:alg-doesnt-fail-trivially} and~\ref{lem:red-not-adj} respectively).  

\begin{lemma}\label{lem:H-edge-prob}
  The probability that the edge set of $H$ contains a fixed set $F$ of edges with $\sabs{F}=\ell$ is at most
  \begin{equation*}
    O\biggl(  \ell! \binom{m}{\ell}\,\frac{\log^{2\ell} n}{n^{2\ell}} \biggr).
  \end{equation*}
\end{lemma}
\begin{proof}
  Recall that $n$ denotes the number of vertices of $H$, which is equal to the number of chunks in $\widetilde G_C$.  This is equal to the number of red
  vertices at the end of Phase~II, which is $\Theta(h^{\nfrac13})$.  The number $m$ of edges of $H$ is equal to the number $p(T)$ of pale vertices after
  termination of Phase~II, which is $O(\log^2 h)$.  The edges come in six different types, depending on which fibre $V_y$, $y\in S$, contains the corresponding
  pale vertex, and also which fibres contain the end-vertices of the two non-exposed edges adjacent to the pale vertex.
  
  For each edge of~$H$, one by one, we draw the two end-vertices one by one.  An edge corresponding to a pale vertex $v$ of $\widetilde G$ connects two fixed
  vertices of~$H$ if the two yet unexposed edges incident to~$v$ end turn out to be contained in the chunks corresponding to the fixed vertices of~$H$.  Since
  the sizes of the chunks are a.a.s.\ $O(\nfrac{h\log^2n}{n})$ by Lemma~\ref{lem:max-gaps-red}, and the number of possible neighbors of~$v$ is between $h$ and
  $h-n-m+O(1) = \Theta(h)$, the probability that the edge of~$H$ connects the two fixed vertices is $O(\frac{\log^2 n}{n^2}$.
  
  From this, the statement of the lemma follows.
\end{proof}


Now we adapt the classical first-moment calculation to prove that there are no cycles in $H$, and therefore, no cycles in the graph induced on uncoloured
vertices in Step~(7).

\begin{lemma}\label{lem:no-circle}
  A.a.s.\ $H$ contains no cycles.
\end{lemma}
\begin{proof}
  By Lemma~\ref{lem:H-edge-prob}, the expected number of cycles of length $\ell \ge 1$ is 
  \begin{equation*}
    \sum_\sstack{C\text{ cycle}\\\sabs{C}=\ell} \Prb[ C \subseteq H ] 
    =
    O\biggl(  \binom{n}{\ell}\,\ell!\;\binom{m}{\ell}\;\frac{\log^{2\ell} n}{n^{2\ell}} \biggr).
  \end{equation*}
  Summing over all possible values of $\ell$, we obtain an upper bound for the expected number of cycles in $H$: With $t := \nfrac{(C \log^{2}n)}{n}$ for a
  suitable constant $C$, we have
  \begin{multline*}
    \sum_{\ell=1}^m \binom{n}{\ell}\;\ell!\;\binom{m}{\ell}\;\frac{\log^{2\ell} n}{n^{2\ell}}
    \\
    \le
    \sum_{\ell=1}^m \binom{m}{\ell} t^\ell
    =
    -1 + (1+t)^m
    \le
    -1 + e^{m t}
    =
    \\
    =
    -1 + e^{\nfrac{(C \log^4n)}{n}} = o(1).
  \end{multline*}

\end{proof}

\section{Conclusions}
The argument for 3-colourability of random lifts of $K_5\backslash e$ in this manuscript can be extended to a more general class of base graphs.  Let $G:=
G_{k,s}$ be a graph obtained by joining a stable set $S$ of size $s$ to a cycle $C$ of size $k$, where $k\ge 3$ and $s\ge 1$.  For $k=3$ and $s=2$ we recover
$K_5\setminus e$.  The proof of Theorem~\ref{thm:K_5-e} extends with hardly any changes to the following.

\begin{theorem}
  The chromatic number of a random lift of $G_{k,s}$ is a.a.s.\ three.
\end{theorem}

It is known that the chromatic number of random 4-regular graphs (with uniform distribution) is three~\cite{wormald}.  Even though random lifts of $K_{d+1}$
have some similarity to random $d$-regular graphs, adapting the methods of the latter to obtain results for random lifts of $K_{d+1}$ appears to be a challenging
task.

\section*{Appendix: Distribution of the gaps between $n$ points drawn in $[0,1]$}

As mentioned above, Lemma~\ref{lem:distrib-gaps-points-interval} is a special case of the Bapat-Beg theorem.  For the sake of completeness, we give an
elementary proof.

\begin{proof}[Proof of Lemma~\ref{lem:distrib-gaps-points-interval}]
  Clearly, $\min(X_1,\dots,X_n)$ has cumulative distribution function $t\mapsto 1-(1-t)^n$.  This settles the easy cases when $k=0$ or, $k=n$.

  Partitioning $\bigotimes_{j=1}^n [0,1]$ into $n!$ sets we need to compute
  \begin{equation}\label{eq:distrgaps:mainprob}
    \Prb[S'_{k+1}-S'_k \le a] = n! \int_{\RR^n} 1_{\{0\le\pr_1\le\dots\le\pr_n\le 1\}} 1_{\{\pr_k \le \pr_{k+1} \le \pr_k + a\}} \,d\lambda^n.
  \end{equation}
  Denoting
  \begin{equation*}
    v(\ell,t) := \int_{\RR^\ell} 1_{\{0\le\pr_1\le\dots\le\pr_\ell\le t\}} \,d\lambda^n = \frac{t^\ell}{\ell!}
  \end{equation*}
  we have that~\eqref{eq:distrgaps:mainprob} is equal to 
  \begin{multline}\label{eq:distrgaps:2-vars}
    \int_0^1\int_0^1 v(s,k-1) v(1-t,n-k-1) 1_{s\le t \le s+a} \,dt \,ds 
    =\\=
    \int_0^1 v(s,k-1) \int_s^{\min(1,s+a)}  v(1-t,n-k-1) \,dt \,ds 
    =\\=
    \frac{1}{(k-1)! (n-k-1)!} \int_0^1 s^{k-1} \int_s^{\min(1,s+a)}  (1-t)^{n-k-1} \,dt \,ds 
  \end{multline}
  We evaluate the inner integral
  \begin{multline*}
    \int_s^{\min(1,s+a)}  (1-t)^{n-k-1} \,dt
    =\\=
    \int_s^{\min(1,s+a)}  (1-t)^{n-k-1} \,dt  
    =\\=
    \begin{cases}
      \frac{1}{n-k} (1-s)^{n-k}                               & \text{ if }s \le 1-a \\
      \frac{1}{n-k} (1-s)^{n-k} - \frac{1}{n-k} (1-a-s)^{n-k} & \text{ if }s \ge 1-a.
    \end{cases}
  \end{multline*}
  Then the integral in~\eqref{eq:distrgaps:2-vars} (without the factorial factor) becomes
  \begin{multline*}
    \frac{1}{n-k} \int_0^1 s^{k-1} (1-s)^{n-k} \,ds  - \frac{1}{n-k} \int_0^{1-a} s^{k-1} (1-a-s)^{n-k}
    =\\=
    - \frac{(k-1)! (n-k-1)! }{n!} (0-1) + \frac{(k-1)! (n-k-1)! }{n!} (0- (1-a)^n)
    =\\=
    \frac{(k-1)! (n-k-1)! }{n!}( 1 - (1-a)^n ).
  \end{multline*}
  Hence, \eqref{eq:distrgaps:mainprob} is equal to
  \begin{equation*}
    n! \frac{1}{(k-1)! (n-k-1)!} \frac{(k-1)! (n-k-1)! }{n!}( 1 - (1-a)^n ) = 1 - (1-a)^n.
  \end{equation*}
  
\end{proof}

\end{document}